\documentclass[a4paper,11pt]{article}
\pdfoutput=1

\usepackage[english]{babel}
\usepackage{cite}
\usepackage{amsfonts}
\usepackage{latexsym}
\usepackage{amsthm}
\usepackage{amsopn}
\usepackage{amsmath}
\usepackage[all]{xy}
\usepackage{verbatim}
\usepackage{amssymb}
\usepackage{mathrsfs}
\usepackage{float}
\usepackage[active]{srcltx}
\usepackage{fullpage}
\usepackage{manyfoot}
\usepackage{amscd}
\usepackage[dvips]{graphics}
\usepackage[dvips]{graphicx}
\usepackage{amscd}
\usepackage{color}
\usepackage{cite}
\usepackage{tikz}
\usepackage{dsfont}
\usepackage{capt-of}
\usetikzlibrary{matrix}
\usepackage[hmargin=3cm, vmargin=3cm]{geometry}
\usepackage{color, xcolor}
\usepackage[backref=page]{hyperref}

\renewcommand*{\backrefalt}[4]{%
    \ifcase #1 (Not cited.)%
    \or       (Cited on page {\textcolor{black}{#2}}.)
    \else     (Cited on pages {\textcolor{black}{#2}}.)
    \fi}

\renewcommand{\boldsymbol}{\mathbf}

\newtheorem{theorem}{Theorem}[section]
\newtheorem{proposition}[theorem]{Proposition}
\newtheorem{lemma}[theorem]{Lemma}

\newtheorem*{main}{Main Theorem}

\theoremstyle{definition}
\newtheorem{remark}[theorem]{Remark}

\newcommand{\lr}{\longrightarrow}
\newcommand{\M}{\mathcal{M}}

\newcommand{\ZZ}{\mathbb{Z}}

\newcommand{\CC}{\mathbb{C}}
\newcommand{\PP}{\mathbb{P}}
\newcommand{\oo}{\mathcal{O}}
\renewcommand{\to}{\longrightarrow}

\title{A family of surfaces with $p_g=q=2, \, K^2=7$ and Albanese map of degree $3$}
\author{Roberto Pignatelli and Francesco Polizzi}
\date{}
\begin{document}
\maketitle
%


\begin{abstract}
We study a family of surfaces of general type with $p_g=q=2$ and $K^2=7$, originally constructed by Cancian and Frapporti by using the Computer Algebra System \verb|MAGMA|. We provide an alternative, computer-free construction of these surfaces, that allows us to describe their Albanese map and their moduli space.
\end{abstract}


\Footnotetext{{}}{2010 \textit{Mathematics Subject Classification}:
14J29, 14J10}

\Footnotetext{{}} {\textit{Keywords}: Surface of general type,
Albanese map, triple cover}


\setcounter{section}{-1}

\section{Introduction} \label{sec:intro}

In recent years, the work of several authors on the classification of \emph{irregular} algebraic surfaces (that is, surfaces $S$ with $q(S) >0$) produced a considerable amount of results, see for example the survey papers \cite{BaCaPi06, MP12} for a detailed bibliography on the subject.

In particular, surfaces of general type with $\chi(\mathcal{O}_S)=1$, that is, $p_g(S)=q(S)$ were investigated. For these surfaces, \cite[Th\'eor\`eme 6.1]{Deb81} implies $p_g \leq 4$. Surfaces with $p_g=q=4$ and $p_g=q=3$ are nowadays completely classified, see \cite{Be82, CaCiML98, HP02, Pir02}. On the other hand, for the the case $p_g=q=2$, which presents a very rich and subtle geometry, we have so far only a partial understanding of the situation; we refer the reader to \cite{PePol13a, PePol13b, PePol14} for an account on this topic and recent results. 

As the title suggest, in this paper we consider a family $\M$ of minimal surfaces of general type with $p_g=q=2$ and $K^2=7$. The existence of such surfaces was originally established in \cite{CanFr15} with the help of the Computer Algebra System \verb|MAGMA| \cite{BCP97}; the present work provides an alternative, computer-free construction of them, that allows us to describe their Albanese map and their moduli space. 

Our results can be summarized as follows, see Theorem \ref{thm:main}.
\begin{main}
There exists a $3$-dimensional family $\mathcal{M}$ of minimal surfaces of general type with $p_g=q=2$ and $K^2=7$ such that, for all elements $S \in \mathcal{M}$, the canonical class $K_S$ is ample and the Albanese map $\alpha \colon S \to A$ is a generically finite triple cover of a principally polarized abelian surface $(A, \, \Theta)$, simply branched over a curve $D_A$ numerically equivalent to $4 \Theta$ having an ordinary sextuple point and no other singularities.
The family $\mathcal{M}$ provides a generically smooth, irreducible, open and normal subset of the Gieseker moduli space $\mathcal{M}^{\mathrm{can}}_{2, \, 2, \, 7}$ of canonical models of minimal surfaces of general type with $p_g=q=2$ and $K^2=7$. 
\end{main}
In particular, this means that $\mathcal{M}$ provides a dense open set of a generically smooth, irreducible component of $\mathcal{M}^{\mathrm{can}}_{2, \, 2, \, 7}$. Furthermore, denoting by $\mathcal{M}_2$ the coarse moduli space of curves of genus $2$, there exists a quasi-finite, surjective morphism $\varsigma \colon \mathcal{M} \to \mathcal{M}_2$ of degree $40$ (see Proposition \ref{prop:def-S-C2}). 

Let us explain now how the paper is organized. In Section \ref{sec:prod-quot} we explain our construction in detail and we compute the invariants of the resulting surfaces (Proposition \ref{prop:invariants-S}); moreover we study their Albanese map, giving a precise description of its image and of its branch curve (Proposition \ref{prop: branch}). It is worth pointing out that the general surface $S$ contains no irrational pencils (Proposition \ref{prop:no-irr-penc}). 

Section \ref{sec:moduli} is devoted to the study of the first-order deformations of the surfaces in $\M$ and to the description of the corresponding subset in $\M^{\mathrm{can}}_{2, \, 2, \, 7}$. A key point in our analysis is showing that for all elements in $S \in \M$ we have $h^1(S, \, T_S)=3$, see 
Proposition \ref{prop:h1T=3}.

Since the degree of the Albanese map is in this case a topological invariant (Proposition \ref{prop.degree.alb}), it follows that these surfaces lie in a different connected component of the moduli space than the only other known example  with the same invariants, namely the surface with $p_g=q=2$ and $K^2=7$ constructed in \cite{Ri15}, whose Albanese map is a generically finite \emph{double} cover of an abelian surface with polarization of type $(1, \,2)$, see Remark \ref{rem:rito}. Hence the family $\mathcal{M}$ provides a substantially new piece in the fine classification of minimal surfaces of general type with $p_g=q=2$.

\bigskip

\textbf{Acknowledgments.} F. Polizzi was partially supported by
Progetto MIUR di Rilevante Interesse Nazionale \emph{Geometria delle
Variet$\grave{a}$ Algebriche e loro Spazi di Moduli} and by
GNSAGA-INdAM.

He thanks the \verb|MathOverflow| users abx and D. Speyer for their help with the proof of Proposition \ref{prop:allcurves}, and J. Starr for useful suggestions and for pointing out the reference \cite{Mat62}. He is also grateful to T. Gentile for his help with Figure 1.

Roberto Pignatelli was partially supported by the project Futuro in Ricerca 2012 \emph{Moduli Spaces and Applications}.  He is a member of GNSAGA-INdAM. He wishes to thank S\"onke Rollenske for his help with the proof of Lemma \ref{lem:restriction}.

Both authors thank the organizers of the intensive research period Algebraic Varieties and their Moduli at the Centro di Ricerca Matematica Ennio de Giorgi (Pisa) for the invitation and the hospitality in the period May-June 2015, when this project started.

They are also grateful to the referee for many detailed comments and suggestions that considerably improved the presentation of these results.

\bigskip

\textbf{Notation and conventions.} We work over the field
$\mathbb{C}$ of complex numbers. 
By \emph{surface} we mean a projective, non-singular surface $S$,
and for such a surface $K_S$ denotes the canonical
class, $p_g(S)=h^0(S, \, K_S)$ is the \emph{geometric genus},
$q(S)=h^1(S, \, K_S)$ is the \emph{irregularity} and
$\chi(\mathcal{O}_S)=1-q(S)+p_g(S)$ is the \emph{Euler-Poincar\'e
characteristic}.

If $C$ is a smooth curve, we identify $\textrm{Pic}^0(C)$ with the Jacobian variety $J(C)$ by means of the canonical isomorphism provided by the Abel-Jacobi map, see \cite[Theorem 11.1.3]{BL94}. Furthermore, we write $\textrm{Sym}^n(C)$ for the $n$-th symmetric product of $C$. 

Given a finite group $G$ acting on a vector space $V$, we denote by $V^{G}$ the $G$-invariant subspace.


\section{The construction} \label{sec:prod-quot}

Let $V_2$ and $V_3$ be the two hypersurfaces of $\PP^3$ defined by
\begin{equation} \label{eq:quad-cub}
V_2:=\{x_2x_3 + r(x_0, \, x_1)=0\}, \quad V_3:=\{x_2^3+x_3^3+s(x_0, \, x_1)=0\},
\end{equation}
where $r, \, s \in \mathbb{C}[x_0, \, x_1]$ are general homogeneous forms of degree $2$ and $3$, respectively. Then $C_4:=V_2 \cap V_3$ is a smooth, canonical curve of genus $4$.  Denoting by $\xi$ a primitive third root of unity, we see that $C_4$  admits a free action of the cyclic group $\langle \xi \rangle \cong \ZZ/3 \ZZ$, defined by
\begin{equation} \label{eq:action-xi}
\xi \cdot [x_0: x_1:x_2:x_3] = [x_0: x_1: \xi x_2: \xi^2 x_3]
\end{equation}
and the quotient $C_2 := C_4/ \langle \xi \rangle$ is a smooth curve of genus $2$. 

\begin{proposition} \label{prop:allcurves}
All \'etale Galois triple covers of a smooth curve of genus $2$ can be obtained in this way.
\end{proposition}
\begin{proof}
Let $C_2$ be any smooth curve of genus $2$ and choose 
any \'etale $\mathbb{Z}/3 \mathbb{Z}$-cover $c \colon C_4 \to C_2$. Thus $C_4$ is a smooth curve of genus $4$ and we can choose a fixed-point free automorphism $\varphi \colon C_4 \to C_4$ generating the Galois group of the cover.

The curve $C_4$ cannot be hyperelliptic, otherwise its ten Weierstrass points would be an invariant set by any automorphism, which is impossible because any orbit of $c$ consists of three distinct points. Hence the canonical divisor $K_{C_4}$ is very ample and defines an embedding of $C_4$ in $\mathbb{P}^3= \mathbb{P}H^0(C_4, \, K_{C_4})$, whose image (that we still denote by $C_4$) is the complete intersection of a (uniquely determined) quadric hypersurface $V_2$ and a cubic hypersurface $V_3$. It remains to show that we can choose $V_2$ and $V_3$ as in \eqref{eq:quad-cub}.

Pushing down the canonical line bundle of $C_4$ to $C_2$ gives a decomposition of $H^0(C_4, \, K_{C_4})$  into $\mathbb{Z}/3 \mathbb{Z}$-eigenspaces, namely
\begin{equation} \label{eq:KC4}
H^0(C_4, \, K_{C_4})=
H^0(C_2, \, K_{C_2}) \oplus
H^0(C_2,\, K_{C_2} +\eta) \oplus
H^0(C_2, \, K_{C_2}+2\eta)
\end{equation}
where $\eta$ is a non-trivial, $3$-torsion divisor on $C_2$. The first summand in \eqref{eq:KC4} has dimension $2$, whereas the others have dimension $1$; so we can choose a basis $x_0$,  $x_1$, $x_2$, $x_3$ of $H^0(C_4, \, K_{C_4})$ such that $x_0$, $x_1$ generate $H^0(C_2, \, K_{C_2})$ whereas $x_2$ and $x_3$ generate $H^0(C_2,\, K_{C_2} +\eta)$ and $H^0(C_2, \, K_{C_2}+ 2 \eta)$, respectively. This means that, using homogeneous coordinates $[x_0: x_1: x_2: x_3]$ in $\mathbb{P}^3$, the action of $\mathbb{Z}/3 \mathbb{Z}= \langle \xi \rangle$  can be written as in \eqref{eq:action-xi}. 

We start by looking at the invariant quadrics in the homogeneous ideal of $C_4$. There are four invariant monomials of degree $2$, namely 
\begin{equation} \label{eq:monom-V2}
x_0^2, \; x_0x_1, \; x_1^2, \; x_2x_3,
\end{equation}   
hence the invariant subspace $(\textrm{Sym}^2H^0(C_4, K_{C_4}))^{\langle \xi \rangle}$ of $\textrm{Sym}^2H^0(C_4, K_{C_4})$ has dimension $4$. On the other hand, the subspace of invariant quadrics in the homogeneous ideal of $C_4$ is the kernel of the surjective map
\begin{equation*}
(\textrm{Sym}^2 H^0(C_4, K_{C_4}))^{\langle \xi \rangle} \to H^0(C_4, \, 2K_{C_4})^{\langle \xi \rangle} \cong H^0(C_2, \, 2K_{C_2})\cong \CC^3,
\end{equation*}
hence it has dimension $1$. In other words, the unique quadric 
$V_2$ containing $C_4$ is invariant, hence the polynomial defining $V_2$ is a linear combination of the monomials in \eqref{eq:monom-V2}. The coefficient of $x_2x_3$ cannot vanish, or $V_2$ would be reducible, so $V_2$ is as in \eqref{eq:quad-cub}.

Let us look now at the invariant cubics in the homogeneous ideal of $C_4$. There are eight invariant monomials of degree $3$, namely
\begin{equation*}
x_0^3, \; x_0^2x_1, \; x_0x_1^2, \; x_1^3, \;  x_0x_2x_3, \; x_1x_2x_3, \; x_2^3, \; x_3^3, 
\end{equation*} 
hence the invariant subspace $(\textrm{Sym}^3H^0(C_4, K_{C_4}))^{\langle \xi \rangle}$ of $\textrm{Sym}^3H^0(C_4, K_{C_4})$ has dimension $8$. On the other hand, the subspace of invariant cubics in the homogeneous ideal of $C_4$ is the kernel of the surjective map
\begin{equation*}
(\textrm{Sym}^3 H^0(C_4, K_{C_4}))^{\langle \xi \rangle} \to H^0(C_4, \, 3K_{C_4})^{\langle \xi \rangle} \cong H^0(C_2, \, 3K_{C_2})\cong \CC^5,
\end{equation*}
hence it has dimension $3$. In particular, this implies that the general invariant cubic hypersurface $V_3$ containing $C_4$ is not a multiple of the quadric $V_2$. Adding suitable scalar multiples of $x_0V_2$ and $x_1V_2$ in order to get rid of the monomials $x_0x_2x_3$  and $x_1x_2x_3$, and changing coordinates by multiplying $x_2$ and $x_3$ by suitable constants we obtain an equation of $V_3$ as in $\eqref{eq:quad-cub}$ and we are done.
\end{proof}

Let us consider now the product $C_4 \times C_4 \subset \PP^3 \times \PP^3$, and write $\boldsymbol{x}=[x_0: x_1: x_2: x_3]$ for the homogeneous coordinates in the first factor and $\boldsymbol{y}=[y_0: y_1: y_2: y_3]$ for those in the second factor. Then the action of $\langle \xi \rangle$ on $C_4$ induces an action of $H := \langle \xi_x, \, \xi_y, \, \sigma \rangle$ on $C_4 \times C_4$, where
\begin{equation*}
\xi_x(\boldsymbol{x}, \, \boldsymbol{y}) := (\xi \cdot \boldsymbol{x}, \, \boldsymbol{y}), \quad  \xi_y(\boldsymbol{x}, \, \boldsymbol{y}):= (\boldsymbol{x}, \,  \xi \cdot \boldsymbol{y}), \quad \sigma(\boldsymbol{x}, \, \boldsymbol{y}) := (\boldsymbol{y}, \, \boldsymbol{x}).
\end{equation*}
Clearly $\xi_x$ and $\xi_y$ commute, whereas $\sigma \xi_x =\xi_y\sigma$ and $\sigma \xi_y =\xi_x\sigma$, so $H$ is a semi-direct product of the form 
\begin{equation*}
H = \langle \xi_x, \, \xi_y \rangle \rtimes \langle \sigma \rangle \cong (\mathbb{Z}/ 3 \ZZ)^2 \rtimes \ZZ/2 \ZZ.
\end{equation*}
In particular,  $|H|=18$ and every element $h \in H$ can be written in a unique way as $h=\sigma^k \xi_x ^i \xi_y ^j$, where $k \in \{0, \, 1\}$ and $i, \, j \in \{0, \, 1, \, 2\}$. 

\begin{lemma} \label{lem: stabilizers-product-quotient}
The non-trivial elements of $H$ having fixed points on $C_4 \times C_4$ are precisely the three elements of order $2$
\begin{equation*}
h_i := \sigma \xi_x^i \xi_y ^{3-i}, \quad i=0, \, 1, \, 2. 
\end{equation*}
More precisely, the element $h_i$ fixes pointwise the smooth curve 
\begin{equation*} 
\Gamma_i :=\{(\boldsymbol{x}, \, \xi^i \cdot \boldsymbol{x}) \, | \, \boldsymbol{x} \in C_4 \}, 
\end{equation*} 
that is, the graph of the automorphism of $C_4$ defined by $\boldsymbol{x} \mapsto \xi^i \cdot \boldsymbol{x}$. The three curves $\Gamma_0$,  $\Gamma_1$ and $\Gamma_2$ are isomorphic to $C_4$, pairwise disjoint and their self-intersection equals $-6$. 
\end{lemma}
\begin{proof} Let $h=\sigma^k \xi_x ^i \xi_y ^j$ be an element of $H$.
If $k=0$ then $h (\boldsymbol{x}, \, \boldsymbol{y})= (\xi^i \cdot \boldsymbol{x}, \, \xi^j \cdot \boldsymbol{y})$ so, since the action of $\xi$ on $C_4$ is free, $h$ has fixed points if and only if it is trivial. Thus we can assume $k=1$, in which case we have
\begin{equation*}  
\sigma \xi_x^i \xi_y^j (\boldsymbol{x}, \, \boldsymbol{y})= (\xi^j \cdot \boldsymbol{y}, \, \xi^i \cdot \boldsymbol{x}). 
\end{equation*}
Hence $(\boldsymbol{x}, \, \boldsymbol{y})$ is a fixed point for $h$ 
if and only if $i+j \equiv 0\, (\textrm{mod }3)$ and $\boldsymbol{y}= \xi^i \cdot \boldsymbol{x}$, that is $(\boldsymbol{x}, \, \boldsymbol{y}) \in \Gamma_i$. 
 
A straightforward computation using the relations $\sigma^2=1$ and $\xi_x \sigma = \sigma \xi_y$ shows that the order of $h_i$ is $2$.

The curve $\Gamma_0$ is the diagonal of $C_4 \times C_4$, hence it is isomorphic to $C_4$ and satisfies $(\Gamma_0)^2 = 2-2g(C_4)=-6$. The same is true for the curves $\Gamma_1$ and $\Gamma_2$, because they are the translate of $\Gamma_0$ by the action of $\xi_y$ and $\xi_x$, respectively. Finally, $\Gamma_i$ and $\Gamma_j$ are disjoint for $i \neq j$, because  $\xi$ acts freely on $C_4$.
\end{proof}

Lemma \ref{lem: stabilizers-product-quotient} implies that the quotient map $C_4\times C_4 \to (C_4 \times C_4)/H$ is ramified exactly over the three curves $\Gamma_i$, with ramification index $2$ on each of them. We factor such a map through the quotient by the normal abelian subgroup $\langle \xi_x, \,  \xi_y \rangle \cong (\ZZ/3\ZZ)^2$. This subgroup acts separately on the two factors, whereas $\sigma$ exchanges them, so we get 
\begin{equation*}
(C_4 \times C_4)/ \langle \xi_x, \,  \xi_y \rangle \cong C_2 \times C_2, \quad  (C_4 \times C_4)/H \cong \textrm{Sym}^2 (C_2).   
\end{equation*}
Thus the surface $B =(C_4 \times C_4)/H$ contains a unique rational curve, namely the $(-1)$-curve $E$ corresponding to the unique $g_2^1$ of $C_2$. Denoting by $\pi \colon B \to  A$ the blow-down of $E$, we see that $A$ is an abelian surface isomorphic to the Jacobian variety $J(C_2)$.

\begin{remark} \label{rem:all-jacobians}
Because of Proposition \ref{prop:allcurves}, all Jacobians of smooth curves  of genus $2$ can be obtained in this way.
\end{remark}
\bigskip
Let us denote now by $\xi_{xy}$ the element $\xi_x \xi_y$ and set $G:=\langle \xi_{xy}, \, \sigma \rangle$; then $G$ is a non-normal, abelian subgroup of $H$, isomorphic to $\ZZ / 2 \ZZ \times \ZZ / 3 \ZZ$. Setting
\begin{equation*}
T := (C_4 \times C_4)/ \langle \xi_{xy} \rangle, \quad S := (C_4 \times C_4) /G, 
\end{equation*}
and writing $t \colon C_4 \times C_4 \to T$ and $f \colon C_4 \times C_4 \to S$ for the corresponding projection maps, we have the following commutative diagram:
\begin{equation} \label{diag:quotients}    
\begin{split}
   \xymatrix{ 
        C_4 \times C_4\ar@/{}_{1pc}/[ddr]_{f} \ar@/{}^{1pc}/[drr] \ar[dr]^{t}\\ 
        & T=(C_4 \times C_4)/ \langle \xi_{xy} \rangle \ar[d]^{u} \ar[r]^-{\gamma} & (C_4 \times C_4)/ \langle \xi_x, \, \xi_y \rangle \cong C_2 \times C_2 \ar[d]^{v} \\  & S=(C_4 \times C_4)/G \ar[r]^-{\beta} \ar[dr]_-{\alpha} & (C_4 \times C_4)/H =B  \cong \textrm{Sym}^2 (C_2)  \ar[d]^{\pi} \\ & & A.}
     \end{split}
\end{equation}
The morphism $u \colon T \to S$ is a double cover, induced by the involution $\sigma$ exchanging the two coordinates in $C_4 \times C_4$. 

We first compute the invariants of $T$.
\begin{lemma} \label{lem:invariants-T}
The surface $T$ is minimal of general type with
\begin{equation*}
p_g(T)=6, \quad q(T)=4, \quad K_T^2=24.
\end{equation*}
\end{lemma}
\begin{proof}
By standard calculations we have
\begin{equation*} 
p_g(C_4 \times C_4)= 16, \quad q(C_4 \times C_4)=8, \quad K_{C_4 \times C_4}^2= 72. 
\end{equation*}
The group $\langle \xi_{xy}{} \rangle \cong \ZZ / 3\ZZ$ acts diagonally and freely on $C_4 \times C_4$, hence $T$ is a so-called \emph{quasi-bundle}, see for instance \cite[Section 3]{Pol09}. Therefore we obtain
\begin{equation*} 
K_T^2 = \frac{1}{3} K_{C_4 \times C_4}^2=24, \quad \chi(\mathcal{O}_T)=\frac{1}{3}\chi(\mathcal{O}_{C_4 \times C_4})= 3, \quad q(T)=g(C_2) + g(C_2)=4, 
\end{equation*}
so $p_g(T)=6$. Note that by Noether's formula this implies $c_2(T)=12$. Finally, $T$ is a minimal surface of general type because it a finite, \'{e}tale quotient of the minimal surface of general type $C_4 \times C_4$.  
\end{proof}

The three curves $\Gamma_i \subset C_4 \times C_4$ are $\xi_{xy}$-invariant, hence their images $\Sigma_i :=t(\Gamma_i) \subset T$ are three curves isomorphic to $C_2$ and such that $(\Sigma_i)^2=\frac{1}{3}(\Gamma_i)^2=-2$. Moreover, the curve $\Gamma_0$ is also $\sigma$-invariant, whereas $\Gamma_1$ and $\Gamma_2$ are switched by the action of $\sigma$. Then $D_S:=u(\Sigma_0)$ and $R:= u(\Sigma_1)= u(\Sigma_2)$ are two disjoint curves in $S$, both isomorphic to $C_2$, such that $(D_S)^2=-4$ and $R^2=-2$. Note that $D_S$ is the branch locus of the double cover $u \colon T \to S$. 

We can now compute the invariants of $S$.

\begin{proposition} \label{prop:invariants-S}
The surface $S$ is a minimal surface of general type with
\begin{equation*}
p_g(S)=2, \quad q(S)=2, \quad K_S^2=7.
\end{equation*}
The morphism $\beta \colon S \to B$ is a non-Galois triple cover, simply ramified over $R$ and simply branched over the diagonal $D_B \subset  B$. Finally, $S$ contains no rational curves $($in particular, $K_S$ is ample$)$ and contains a smooth elliptic curve, namely 
$Z:=\beta^* E$ $($which satisfies $Z^2=-3)$.
\end{proposition}  
\begin{proof}
We start by proving the last claim. The two smooth curves $D_B$ and $E$ intersect transversally at the six points corresponding to the six Weierstrass points of $C_2$, so $Z:=\beta^* E$ is a smooth, irreducible curve of genus $1$ contained in $S$, such that
\begin{equation} \label{eq:ZR}
ZR = (\beta^*E).R = E.(\beta_*R) = E D_B =6. 
\end{equation}

On the other hand, $S$ does not contain any rational curve. Otherwise, such a curve would map would map onto $E$ via $\beta \colon S \to B$, impossible because we have seen that $\beta^*E$ is smooth of genus $1$.

Since the double cover $u \colon T \to S$ is branched over the curve $D_S$, it follows that $D_S$ is $2$-divisible in $\textrm{Pic}(S)$ and moreover  
\begin{equation*}
24=K_T^2=2\bigg(K_S+\frac{1}{2}D_S \bigg)^2.
\end{equation*}
Using $(D_S)^2=-4$ and $K_SD_S=6$, we find $K_S^2=7$. Since $S$ does not contain any rational curve and $K_S^2>0$, we deduce that $S$ is a minimal surface of general type with ample canonical class.

Now, as $K_B=\pi^*K_A+E=E$, the Riemann-Hurwitz formula yields  
\begin{equation} \label{eq:K_S}
K_S = \beta^* K_B + R = Z + R,
\end{equation}
and this allows us to compute $Z^2$. In fact, using \eqref{eq:ZR} and \eqref{eq:K_S}, we can write
\begin{equation*}
7 = K_S^2 = Z^2 + 2ZR + R^2 = Z^2 +10, 
\end{equation*}
that is $Z^2=-3$.

Next, denoting by $\chi_{\textrm{top}}$ the topological Euler number, we have 
\begin{equation*}
\begin{split}
\chi_{\textrm{top}} (S-D_S-R)& = \frac{1}{2} \chi_{\textrm{top}} (T - \Sigma_0 - \Sigma_1 - \Sigma_2) \\
& = \frac{1}{2} (c_2(T) - \chi_{\textrm{top}}(\Sigma_0) - \chi_{\textrm{top}}(\Sigma_1) - \chi_{\textrm{top}}(\Sigma_2))=\frac{1}{2}(12-3(-2)) =9,
\end{split}
\end{equation*}
so
\begin{equation*}
c_2(S)=\chi_{\textrm{top}} (S)=\chi_{\textrm{top}} (S-D_S-R) + \chi_{\textrm{top}}(D_S) + \chi_{\textrm{top}} (R)= 9 - 2 - 2 =5.
\end{equation*}
Therefore Noether's formula yields $\chi(\mathcal{O}_S)=1$, that is $p_g(S)=q(S)$. 

The existence of the surjective morphism $\alpha \colon S \to A$ implies $q \geq 2$, and since minimal surfaces of general type with $p_g=q\geq 3$ have either $K^2=6$ or $K^2=8$ (see for instance \cite{BaCaPi06}), we deduce $p_g(S)=q(S)=2$.

The morphism $\beta \colon S \to B$ is a non-Galois triple cover, because $G$ is a non-normal subgroup of index $3$ in $H$. Since $t \colon C_4 \times C_4 \to  T$ is \'{e}tale and $u \colon  T  \to S$ is branched over $D_S$, by Lemma  \ref{lem: stabilizers-product-quotient} it follows that $\beta \colon S \to B$ is simply ramified over $R$, and hence simply branched over $\beta(R) = D_B$. 
\end{proof}

\begin{remark} \label{rem:Cancian-Frapporti}
The existence of surfaces $S$ was first established in \cite{CanFr15}, using a computer-aided construction based on \verb|Magma| computations. 
The present paper provides the first computer-free description of them.  Actually, $S$ is a {\it semi-isogenous mixed surface}, namely a quotient of type $(C \times C)/G,$ where $C$ is a smooth curve and $G$ is a finite subgroup of $\mathrm{Aut}(C \times C)$, such that the subgroup $G^0$  of the automorphisms preserving both factors has index $2$ and acts freely. In fact, with our previous notation 
\begin{equation*}
C=C_4, \quad G= \langle \xi_{xy}, \, \sigma \rangle, \quad G^0= \langle \xi_{xy} \rangle. 
\end{equation*}
The paper \cite{CanFr15} provides a detailed study of semi-isogenous mixed surfaces, showing, among other things, that they are smooth and how to compute their invariants. For instance, \cite[Proposition 2.6]{CanFr15} allows us to prove the equality $q(S)=2$ without exploiting the classification of surfaces with $p_g=q \geq 3$. 
\end{remark}
\bigskip 
Let us now identify the blow-up morphism $\pi \colon B \to A$ with the Abel-Jacobi map 
\begin{equation*}
\textrm{Sym}^2(C_2) \to J(C_2). 
\end{equation*}
If $\Theta$  is the class of a theta divisor in $\textrm{NS}(A)$, let us define the class $\Theta_B:=\pi^* \Theta$ in $\textrm{NS}(B)$. Moreover, let us write $x$ for the class in $\textrm{NS}(B)$ given by the image of the map
\begin{equation*}
C_2 \to \textrm{Sym}^2(C_2), \quad p \mapsto p_0+p
\end{equation*}
where $p_0 \in C_2$ is fixed (such a class does not depend on $p_0$). Then we can prove the following
\begin{lemma} \label{lem:delta}
The equality  
$\pi_* D_B = 4 \Theta$ holds in $\mathrm{NS}(A)$.
\end{lemma}
\begin{proof}
This is a consequence of general results on $g$-fold symmetric products of curves of genus $g$. For instance, \cite[equations (1) and (5)]{Mat62} give in our case the relations 
\begin{equation*}
2 E + D_B = 4x, \quad \Theta_B = E+x
\end{equation*} 
in $\textrm{NS}(B)$, and these in turn imply $D_B=4 \Theta_B - 6E$. So the result follows by applying the push-forward map $\pi_* \colon \textrm{NS}(B) \to \textrm{NS}(A)$.
\end{proof}

The next step consists in describing the Albanese morphism of $S$.
\begin{proposition}\label{prop: branch}
The abelian surface $A$ is isomorphic to $\mathrm{Alb}(S)$ and, up to automorphisms of $A$, the generically finite triple cover $\alpha=\pi \circ \beta \colon S \to A$ coincides with the Albanese morphism of $S$. Furthermore, the only curve contracted by $\alpha$ is $Z$. Finally, $\alpha$ is branched over a divisor $D_A$ numerically equivalent to $4 \Theta$, having an ordinary sextuple point and no other singularities. 
\end{proposition}
\begin{proof}
By the universal property of the Albanese variety (\cite[Chapter V]{Be96}), the morphism $\alpha \colon S \to A$ must factor through the Albanese morphism of $S$; but $\alpha$ is surjective and generically of degree $3$, so it must actually coincide with the Albanese morphism of $S$ up to automorphisms of $A$. Since $\beta$ is a finite morphism, $\alpha$ only contracts the preimage of $E$ in $S$, which is $Z$. The branch locus $D_A$ of $\alpha$ is equal to the image of the diagonal $D_B$ via $\pi \colon B \to A$; since $D_B$ is smooth and intersects $E$ transversally at six points, it follows that $D_A$ has an ordinary sextuple point and no other singularities. Finally, the fact that $D_A$ is numerically equivalent to $4 \Theta$ follows from Lemma \ref{lem:delta}.
\end{proof}
The situation is summarized in Figure \ref{fig:cover} below.
\begin{center}
\begin{tikzpicture}[xscale=-1,yscale=-0.25,inner sep=0.7mm,place/.style={circle,draw=black!100,fill=black!100,thick}] \label{fig:cover} 
\draw (-0.7,-0.5) rectangle (0.6,5);

\draw[red,rotate=92,x=6.28ex,y=1ex] (0.9,-0.85) cos (1,0) sin (1.25,1) cos (1.5,0) sin (1.75,-1) cos (2,0) sin (2.25,1) cos (2.5,0) sin (2.75,-1) cos (3,0) sin (3.25,1) cos (3.5,0) sin (3.6,-0.85);
\draw[rotate=92] (0.5,-0.025) .. controls (1.75,0.035) and (2.75,0.035) .. (4,-0.025);

\draw (-6.7,-0.5) rectangle (-5.4,5);

\draw[red,rotate=92,x=6.28ex,y=1ex,xshift=6,yshift=170] (0.9,-0.85) cos (1,0) sin (1.25,1) cos (1.5,0) sin (1.75,-1) cos (2,0) sin (2.25,1) cos (2.5,0) sin (2.75,-1) cos (3,0) sin (3.25,1) cos (3.5,0) sin (3.6,-0.85);
\draw[rotate=92,xshift=6,yshift=170] (0.5,-0.025) .. controls (1.75,0.035) and (2.75,0.035) .. (4,-0.025);

\draw (-6.7,15.5) rectangle (-5.4,21);

\draw[red,xscale=-1,yscale=-4] (6.05,-4.55) node(A0) [place,scale=0.2]{} to [in=5,out=55,looseness=8mm,loop] () to [in=65,out=115,looseness=8mm,loop] () to [in=125,out=175,looseness=8mm,loop] () to [in=185,out=235,looseness=8mm,loop] () to [in=245,out=295,looseness=8mm,loop] () to [in=305,out=355,looseness=8mm,loop] ();

\draw[xscale=-1,yscale=-4] (0.2,-0.1) node(B0) []{\textcolor{red}{\footnotesize{$R$}}};
\draw[xscale=-1,yscale=-4] (0.25,-1.05) node(0B) []{\footnotesize{$Z$}};

\draw[xscale=-1,yscale=-4] (6.25,-0.1) node(B1) []{\textcolor{red}{\footnotesize{$D_B$}}};
\draw[xscale=-1,yscale=-4] (6.25,-1.05) node(1B) []{\footnotesize{$E$}};

\draw[xscale=-1,yscale=-4] (6.4,-5) node(B2) []{\textcolor{red}{\footnotesize{$D_A$}}};

\draw[xscale=-1,yscale=-4] (-0.9,-0.125) node(C0) []{$S$};
\draw[xscale=-1,yscale=-4] (7,-0.125) node(C1) []{$B$};
\draw[xscale=-1,yscale=-4] (7,-4.125) node(C2) []{$A$};

\draw[->] (-0.9,2.25) -- (-5.2,2.25) node[midway,above] {$\beta$};

\draw[->] (-6.05,5.8) -- (-6.05,14.7) node[midway,right] {$\pi$};

\draw[->] (-0.9,4.8) -- (-5.2,15.5) node[midway,below=3pt] {$\alpha$}; 
\end{tikzpicture}

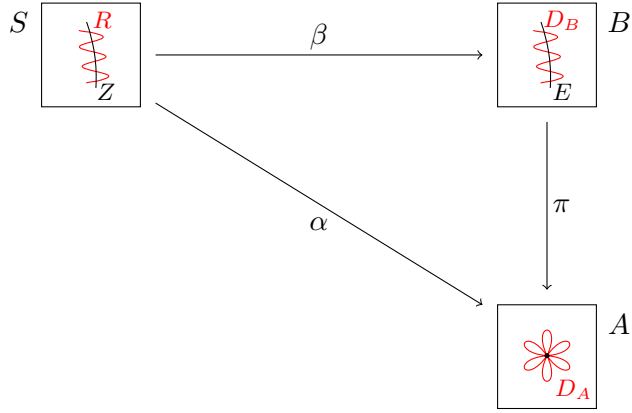
\captionof{figure}{The triple covers $\alpha$ and $\beta$ }
\end{center}
Furthermore, the Stein factorization of $\alpha \colon S \to A$ is described in the diagram 
\begin{equation*} \label{dia.alpha} 
\xymatrix{
S \ar[r]^{c_Z} \ar[dr]_{\alpha} & \tilde{S} \ar[d]^{\tilde{\alpha}} \\
 & A,}
\end{equation*}
where $c_Z \colon S \to \tilde{S}$ is the birational morphism given by the contraction of the elliptic curve $Z$. Since $Z^2=-3$, the normal surface $\tilde{S}$ has a Gorenstein elliptic singularity of type $\tilde{E}_6$, see \cite[Theorem 7.6.4]{Is14}.  

Recall that an \emph{irrational pencil} (or \emph{irrational fibration}) on a smooth, projective surface is a surjective morphism with connected fibres over a curve of positive genus. 

\begin{proposition} \label{prop:no-irr-penc}
The general surface $S$ contains no irrational pencils.
\end{proposition}
\begin{proof}
Assume that $\phi \colon S \to W$ is an irrational pencil on $S$. Since $q(S)=2$, we have either $g(W)=1$ or $g(W)=2$. On the other hand, using the embedding $W \hookrightarrow J(W)$ and the universal property of the Albanese map, we obtain a morphism $A \to J(W)$ whose image is isomorphic to the curve $W$. This rules out the case $g(W)=2$, hence $W$ is an elliptic curve and so $A$ is a non-simple abelian surface. The proof is now complete, because $A$ is isomorphic to the Jacobian variety $J(C_2)$, which is known to be simple for a general choice of $C_2$ (\cite[Theorem 3.1]{Ko76}). 
\end{proof}

\section{The moduli space} \label{sec:moduli}
A projective variety $X$ is called \emph{of maximal Albanese dimension} if its Albanese map $\alpha_X \colon X \to \textrm{Alb}(X)$ is generically finite onto its image. For surfaces of general type with irregularity at least $2$, this is actually a topological property, as shown by the result below.
 
\begin{proposition} \label{prop.degree.alb} Let $S$ be a minimal surface of
general type with $q(S) \geq 2$. If $S$ is of maximal Albanese dimension, then the same holds for any surface which is homeomorphic to $S$. Furthermore, in the case $q(S)=2$ the degree of the Albanese map $\alpha \colon S \to A$ is a topological invariant.
\end{proposition}
\begin{proof}
This follows by the results of  \cite{Ca91}, see for instance \cite[Proposition 3.1]{PePol13b}.
\end{proof}
Proposition \ref{prop.degree.alb} allows us to study
the deformations of $S$ by relating them to those of the flat triple cover $\beta \colon S \to B$. Since the trace map provides a splitting of the short exact sequence 
\begin{equation*}
0 \to \oo_B \to \beta_* \oo_S \to \mathcal{E}_{\beta} \to 0, 
\end{equation*}
we obtain a direct sum decomposition
\begin{equation} \label{eq:Tsch} 
\beta_* \oo_S = \oo_B \oplus \mathcal{E}_{\beta},
\end{equation}
where $\mathcal{E}_{\beta}$ is a vector bundle of rank $2$ on $B$ which satisfies
\begin{equation} \label{eq.coh.E}
h^0(B, \, \mathcal{E}_{\beta})=0, \quad h^1(B, \, \mathcal{E}_{\beta})=0, \quad h^2(B, \, \mathcal{E}_{\beta})=1
\end{equation}
and that, according to \cite{Mir85}, is called the \emph{Tschirnhausen bundle} of $\beta$.

As in \cite[Section 3]{PePol13b} we have a commutative diagram
\begin{equation} \label{dia.TS}
\begin{split}
\xymatrix{ &  &0  \ar[d]  & 0  \ar[d]  & \\
0 \ar[r] & T_S \ar[r]^{d \beta} \ar@{=}[d] & \beta^*T_B
\ar[r] \ar[d]& \mathcal{N}_{\beta} \ar[r] \ar[d]  & 0\\
0 \ar[r] & T_S \ar[r]^{d \alpha}  & \alpha^*T_A \ar[r] \ar[d] &
\mathcal{N}_{\alpha}
\ar[r] \ar[d]& 0 \\
 & & \oo_{Z}(-Z) \ar@{=}[r] \ar[d] & \oo_{Z}(-Z) \ar[d] & 
 \\
&  & 0  & 0 &  }
\end{split}
\end{equation}
whose central column is the pullback via $\beta \colon S \to B$ of a sequence 
\begin{equation} \label{eq.blow-up.tangent}
0 \lr T_B \stackrel{d \pi}{\to} \pi^*T_A \to \oo_E(-E) \to 0,
\end{equation}
see \cite[p. 73]{Sern06}. The normal sheaf $\mathcal{N}_{\alpha}$ of $\alpha \colon S \to A$ is supported on the set of critical points of $\alpha$, namely on the reducible divisor $R + Z$. Analogously, the normal sheaf $\mathcal{N}_{\beta}$ of $\beta \colon S \to B$  is supported on the set of critical points of $\beta$, namely on $R$.  
\begin{lemma} \label{lem:N-beta}
We have
\begin{equation} \label{eq:N-beta}
\mathcal{N}_{\beta} = (N_{R/S})^{\otimes 2} = \mathcal{O}_R(2R).
\end{equation}
Hence all first-order deformations of $\beta \colon S \to B$ leaving $B$ fixed are trivial.
\end{lemma}
\begin{proof}
Since $R$ is smooth, the first statement is a consequence of  \cite[Lemma 3.2]{Rol10}. Furthermore, we observe that $R^2=-2$ implies that the line bundle $\mathcal{N}_{\beta}$ has negative degree on $R$, hence $H^0(R, \, \mathcal{N}_{\beta})=0.$ By \cite[Corollary 3.4.9]{Sern06}, this shows that $\beta \colon S \to B$ is rigid as a morphism with fixed target.
\end{proof}
Note the the last statement of Lemma \ref{lem:N-beta} agrees with the fact that the branch locus $D_B$ of $\beta \colon S \to B$ is a rigid divisor in $B$.  

\begin{lemma}\label{lem:h1Th0N}
We have
\begin{equation*}
h^1(S, \, T_S)=h^0(R + Z, \, \mathcal{N}_{\alpha})+1\geq 3.
\end{equation*}
\end{lemma}
\begin{proof}
Let us write down the cohomology exact sequence associated with the short exact sequence in the central row of (\ref{dia.TS}), recalling first that $S$ is a surface of general type and therefore $h^0(S,T_S)=0$: 
\begin{equation*}
0 \longrightarrow H^0(S, \, \alpha^*T_A)\cong \CC^2  \longrightarrow H^0(R + Z, \, \mathcal{N}_{\alpha})  \longrightarrow  H^1(S, \, T_S) \stackrel{\varepsilon}{\longrightarrow} H^1(S,\, \alpha^*T_A) 
\end{equation*}
Then the claim will follow if we show that $\textrm{rank}(\varepsilon)=3$, and this can be done by using the same argument as in \cite[Section 3]{PePol13b}. 

More precisely, since $T_A$ is trivial and the Albanese map induces an isomorphism $H^1(S, \, {\mathcal O}_S) \cong H^1(A, \, {\mathcal O}_A)$, then  $H^1(S,\, \alpha^*T_A) \cong  H^1(A,\, T_A)$ and we can see $\varepsilon$ as the map $H^1(S, \, T_S) \rightarrow H^1(A, \, T_A)$ induced on first-order deformations by the Albanese map.
By Remark \ref{rem:all-jacobians} the first-order deformations of $S$ dominate the first-order algebraic deformations of $A$, so  $\textrm{rank}(\varepsilon) \geq 3$; on the other hand, the Albanese variety of every deformation of $S$ has to remain algebraic, so $\textrm{rank}(\varepsilon) \leq 3$ and we are done.
\end{proof}

Thus, in order to understand the first-order deformations of $S$, we can study $\mathcal{N}_{\alpha}$.
\begin{lemma}\label{lem:locallyfree}
The sheaf $\mathcal{N}_{\alpha}$ is locally free of rank $1$ on the reducible curve $R + Z$.
\end{lemma}
\begin{proof}
By a standard application of Nakayama's lemma (see for instance \cite[Corollary 5.3.4]{Kempf93}), it suffices to check that the $\mathbb{C}$-vector space $\mathcal{N}_{\alpha, \, x}/\mathfrak{m}_x \mathcal{N}_{\alpha, \, x}$ has dimension $1$ for all $x \in R + Z$, where $\mathfrak{m}_x \subset \mathcal{O}_{R + Z, \, x}$ is the maximal ideal. Equivalently, we will check that the vector bundle map $d \alpha \colon T_S \to \alpha^* T_A $ has rank $1$ at each point $x \in R + Z$. Let us distinguish three cases. 
\begin{itemize}
\item If $x \in R \setminus Z$, then $\alpha$ is locally of the form $(u, \, v)\mapsto (u^2, \, v)$, with $x=(0, \, 0)$ and $R$ given by $u=0$. Then $d \alpha$ is the linear map associated with the matrix
\begin{equation*}
\left(%
\begin{array}{cc}
  2u & 0  \\
  0 & 1 
\end{array}%
\right),
\end{equation*}
which has rank $1$ at the point $x$.
\item
If $x \in Z \setminus R$, then $\alpha$ is locally a smooth blow-up, hence of the form $(u, \, v)\mapsto (uv, \, v)$, where $x=(0, \, 0)$ and $Z$ corresponds to the exceptional divisor, whose equation is $v=0$. Then $d \alpha$ is the linear map associated with the matrix
\begin{equation*}
\left(%
\begin{array}{cc}
  v & u  \\
  0 & 1 
\end{array}%
\right),
\end{equation*}
which has rank $1$ at the point $x$.
\item Finally, if $x \in R \cap Z$ then $\alpha$ is locally the composition of the two maps above, so of the form $(u, \, v)\mapsto (u^2v, \, v)$, where $x=(0, \, 0)$, the curve $R$ corresponds to the locus $u=0$ and the curve $Z$ to the locus $v=0$. Then $d \alpha$ is the linear map associated with the matrix
\begin{equation*}
\left(%
\begin{array}{cc}
  2uv & u^2  \\
  0 & 1 
\end{array}%
\right),
\end{equation*}
which has rank $1$ at the point $x$.
\end{itemize}
This completes the proof.
\end{proof}
We can be more precise and compute the restrictions of $\mathcal{N}_{\alpha}$ to both curves $R$ and $Z$.

\begin{lemma}\label{lem:restriction}
We have
\begin{equation*}
\mathcal{N}_{\alpha|Z} = \oo_Z(-Z), \quad
\mathcal{N}_{\alpha|R} = \oo_R(2R+Z) = \oo_R(K_R).
\end{equation*}
\end{lemma}
\begin{proof}
Let us first apply the functor $\otimes_{\oo_{R + Z}}\oo_Z$ to the exact sequence forming the last column of diagram (\ref{dia.TS}); using \eqref{eq:N-beta}, we get 
\begin{equation*}
\oo_R(2R) \otimes \oo_Z \stackrel{\zeta}{\longrightarrow} \mathcal{N}_{\alpha|Z} \longrightarrow \oo_Z(-Z)  \longrightarrow 0.
\end{equation*}
By Lemma \ref{lem:locallyfree}, the sheaf  $\mathcal{N}_{\alpha|Z}$ is locally free on $Z$; on the other hand, $\oo_R(2R) \otimes \oo_Z$ is a torsion sheaf, hence $\zeta$ is the zero map and so $\mathcal{N}_{\alpha|Z} \cong \oo_Z(-Z)$.

\smallskip

Next, we apply  to the same exact sequence the functor $\otimes_{\oo_{R + Z}}\oo_R$, obtaining
\begin{equation} \label{eq:tensor-OR-1}
\mathcal{T} \stackrel{\tau}{\longrightarrow} \oo_R(2R) \longrightarrow \mathcal{N}_{\alpha|R} \longrightarrow \oo_Z(-Z) \otimes \oo_R \longrightarrow 0.
\end{equation}
Since ${\mathcal T}:=\textrm{Tor}^1_{\mathcal{O}_{R +Z}}(\oo_Z(-Z), \, \oo_R)$ is a torsion sheaf (supported on $R \cap Z$) and $\mathcal{O}_R(2R)$ is locally free on $R$, we deduce that $\tau$ is the zero map and so \eqref{eq:tensor-OR-1} becomes
\begin{equation} \label{eq:tensor-OR-2}
0 \to \oo_R(2R) \longrightarrow \mathcal{N}_{\alpha|R} \longrightarrow \oo_Z(-Z) \otimes \oo_R \longrightarrow 0.
\end{equation}
On the other hand, the curves $R$ and $Z$ intersect transversally at the six Weierstrass points $p_1, \ldots p_6$ of $R$, so we infer 
\begin{equation} \label{eq: OR-OZ}
\oo_Z(-Z) \otimes \oo_R = \oo_Z \otimes \oo_R = \bigoplus_{i=1}^6 \oo_{p_i}.
\end{equation}
Hence \eqref{eq:tensor-OR-2} and \eqref{eq: OR-OZ} yield 
\begin{equation*}
0 {\longrightarrow} \oo_R \longrightarrow \mathcal{N}_{\alpha|R}(-2R) \longrightarrow \bigoplus_{1}^6 \oo_{p_i} \longrightarrow 0,
\end{equation*}
that is the invertible sheaf $\mathcal{N}_{\alpha|R}(-2R)$ has a global section whose divisor is $\sum p_i$. This means $\mathcal{N}_{\alpha|R} \cong \oo_R(2R+\sum p_i)=\oo_R(2R+Z)$. Finally, equation \eqref{eq:K_S} shows that $R+Z$ is a canonical divisor on $S$, so by using adjunction formula we obtain
\begin{equation*}
\oo_R(2R+Z)= \oo_S(K_S+R) \otimes \oo_R = \oo_R(K_R).
\end{equation*}
\end{proof}
We can finally prove
\begin{proposition}\label{prop:h1T=3}
All surfaces $S$ constructed in \emph{Section} \emph{\ref{sec:prod-quot}}  satisfy 
\begin{equation*}
h^1(S, \, T_S)=3.
\end{equation*}
\end{proposition}
\begin{proof}
By Lemma \ref{lem:h1Th0N} it suffices to show the inequality $h^0(R + Z, \, \mathcal{N}_{\alpha})\leq 2$. By \cite[p. 62]{BHPV03} we have a ``decomposition sequence"
\begin{equation*}
0 \to \oo_Z(-R) \to \oo_{R + Z} \to \oo_R \to 0,
\end{equation*}
which gives, tensoring with $\mathcal{N}_{\alpha}$ and using Lemma \ref{lem:restriction},
\begin{equation*}
0 \to \oo_Z(-R-Z) \to \mathcal{N}_{\alpha} \to \oo_R(K_R) \to 0.
\end{equation*}
Since $Z(-R-Z)=-3 <0$, we deduce $H^0(Z, \, \oo_Z(-R-Z))=0$. So 
$H^0(R + Z, \, \mathcal{N}_{\alpha})$ injects into $H^0(R, \, K_R)=\mathbb{C}^2$ and we are done.
\end{proof}
The moduli space of principally polarized abelian surfaces has dimension $3$; moreover,the rigidity of the curve $D_B$ in $B$ implies that the curve $D_A$ has only trivial deformations in $A$. So our surfaces $S$ provide a $3$-dimensional subset $\mathcal{M}$ of the moduli space $\mathcal{M}_{2, \,2, \, 7}^{\textrm{can}}$  of (canonical models of) minimal surfaces of general type with $p_g=q=2$ and $K^2=7$. Because of Proposition \ref{prop:h1T=3}, the corresponding Kuranishi family is smooth; this implies that $\mathcal{M}$ has at most quotient singularities, so it is a normal (and hence generically smooth) open subset of $\mathcal{M}_{2, \,2, \, 7}^{\textrm{can}}.$ In particular, $\mathcal{M}$ provides a dense open set of a generically smooth, irreducible component of this moduli space.

\bigskip
Summing up, we have proven the Main Theorem stated in the introduction, namely

\begin{theorem} \label{thm:main}
There exists a $3$-dimensional family $\mathcal{M}$ of minimal surfaces of general type with $p_g=q=2$ and $K^2=7$ such that, for all elements $S \in \mathcal{M}$, the canonical class $K_S$ is ample and the Albanese map $\alpha \colon S \to A$ is a generically finite triple cover of a principally polarized abelian surface $(A, \, \Theta)$, simply branched over a curve $D_A$ numerically equivalent to $4 \Theta$ having an ordinary sextuple point and no other singularities.
The family $\mathcal{M}$ provides a generically smooth, irreducible, open and normal subset of the Gieseker moduli space $\mathcal{M}^{\mathrm{can}}_{2, \, 2, \, 7}$ of canonical models of minimal surfaces of general type with $p_g=q=2$ and $K^2=7$.
\end{theorem}  

\begin{remark} \label{rem:rito}
By Proposition \ref{prop.degree.alb} the degree of the Albanese map is in our case a topological invariant, so it follows that the surfaces in $\mathcal{M}$ lie in a different connected component of $\mathcal{M}^{\mathrm{can}}_{2, \, 2, \, 7}$ than the only other known example with the same invariants, namely the surface with $p_g=q=2$ and $K^2=7$ constructed in \cite{Ri15}, whose Albanese map is a generically finite \emph{double} cover of an abelian surface with polarization of type $(1, \,2)$. Hence the family $\mathcal{M}$ provides a substantially new piece in the fine classification of minimal surfaces of general type with $p_g=q=2$.
\end{remark}
\bigskip
For every surface $S$ whose isomorphism class $[S]$ belongs to $\mathcal{M}$, the normalization of the branching curve $D_A$ of $\alpha \colon S \to A$ is isomorphic to $C_2$, hence we obtain 
a morphism
\begin{equation*}
\varsigma \colon \M \to \M_2, \quad \varsigma([S]):=[C_2], 
\end{equation*}
where $\mathcal{M}_2$ denotes as usual the coarse moduli space of curves of genus $2$. Note that such a morphism is surjective by Proposition \ref{prop:allcurves}. Correspondingly, we have a morphism of deformation functors, namely
\begin{equation*}
\delta_S \colon \mathrm{Def}_S \to \mathrm{Def}_{C_2}. 
\end{equation*}
The next result clarifies the relation between the deformations of $S$ and those of the curve $C_2$.

\begin{proposition} \label{prop:def-S-C2}
The following holds:
\begin{itemize}
\item[$\boldsymbol{(1)}$] 
$\delta_S \colon \mathrm{Def}_S \to \mathrm{Def}_{C_2}$ is an isomorphism of functors. 
\item[$\boldsymbol{(2)}$] $\varsigma \colon \M \to \mathcal{M}_2$ is a quasi-finite morphism of degree $40$.
\end{itemize}
\end{proposition}
\begin{proof}
$\boldsymbol{(1)}$ Since $\dim \mathcal{M} = \dim H^1(S, \, T_S)=3$, the functor $\mathrm{Def}_S$ is unobstructed; moreover, the functor $\mathrm{Def}_{C_2}$ is  clearly unobstructed, too. 
Proposition \ref{prop:allcurves} implies that the first-order deformations of $S$ dominate the first-order deformations of $C_2$, so the differential map
\begin{equation} \label{eq:differential}
d\delta_S \colon H^1(S, \,T_S) \to H^1(C_2, \, T_{C_2})
\end{equation}
is surjective, and hence it is an isomorphism because $H^1(S, \,T_S)$ and $H^1(C_2, \, T_{C_2})$ have the same dimension. Since $\mathrm{Def}_S$ and $\mathrm{Def}_{C_2}$ are both unobstructed, this shows that $\delta_S$ is an isomorphism of functors, see \cite[Corollary 2.3.7 and Remark 2.3.8]{Sern06}.
\bigskip

$\boldsymbol{(2)}$ We have to show that, for each  $C_2 \in \mathcal{M}_2$, the cardinality of $\varsigma^{-1}(C_2)$ is at most $40$ and that it is exactly $40$ for a general choice of $C_2$.

Remark that, once $C_2$ is fixed, the  \'etale $\mathbb{Z}/3 \mathbb{Z}$-cover $c \colon C_4 \to C_2$ completely determines $S$. Conversely, we claim that, starting from $S$, it is possible to reconstruct the \'etale morphism $c \colon C_4 \to C_2$ up to automorphisms of $C_2$ and $C_4$.  In fact, the subgroup $\xi_{xy}$ is normal in $H$ and the quotient $H / \langle  \xi_{xy} \rangle$ is isomorphic to $S_3$, hence looking at diagram \eqref{diag:quotients} we see that the map
\begin{equation*}
v \circ \gamma \colon T \to B=\mathrm{Sym}^2(C_2)
\end{equation*}     
yields the Galois closure of the triple cover $\beta \colon S \to B$. This shows that $S$ determines the quasi-bundle $T=(C_4 \times C_4)/ \langle \xi_{xy} \rangle$. On the other hand, since the action of $\langle \xi_{xy} \rangle$ on $C_4 \times C_4$ is faithful, if we know $T$ we can reconstruct $C_4$ and the the  \'etale $\mathbb{Z}/3 \mathbb{Z}$-cover $c \colon C_4 \to C_2$ up to automorphisms by using the rigidity result for minimal realizations of surfaces isogenous to a product proven in \cite[Proposition 3.13]{Ca00}.   

Summing up, the cardinality of $\varsigma^{-1}(C_2)$ equals the number of Galois \'etale triple covers $c \colon C_4 \to C_2$ up to equivalence. Here by ``equivalence of covers" we intend commutative diagrams of the form
\begin{equation*}
\begin{split}
   \xymatrix{ 
       C_4 \ar[d]_{c} \ar[r]^-{\varphi_4} & C_4 \ar[d]^-{c} \\  
 C_2 \ar[r]^-{\varphi_2}  & C_2,}
     \end{split}
\end{equation*}
where the horizontal arrows are automorphisms of the corresponding curves. In particular, as explained for instance in \cite{pardini}, if $\varphi_2=\mathrm{id}_{C_2}$ then the number of equivalence classes of Galois triple covers $c \colon C_4 \to  C_2$ coincides with the number of distinct subgroups of order $3$ in $\mathrm{Pic}^0(C_2)$, i.e. with half the number of non-trivial $3$-torsion points, that is $(3^4-1)/2=40$.

On the other hand, if $C_2$ is a general curve of genus $2$ its unique non-trivial automorphism is the hyperelliptic involution, which acts as the multiplication by $-1$ on the group $\mathrm{Pic}^0(C_2)$ and hence trivially on the set of its $40$ subgroups of order $3$. Thus, for a general choice of $C_2$, the fibre $\varsigma^{-1}(C_2)$ consists of exactly $40$ distinct points. 
\end{proof}

\begin{remark} \label{rem:Torelli}
Let us denote by $\mathcal{A}_2$ the coarse moduli space of principally polarized abelian surfaces. It is well-known that the Torelli map 
$\tau_2 \colon \mathcal{M}_2 \to \mathcal{A}_2$, 
sending every curve to its polarized Jacobian, is an immersion, see \cite{OS80}. Thus, composing $\tau_2$ with $\varsigma$, we obtain a generically finite dominant morphism 
$\tau_2 \circ \varsigma \colon \mathcal{M} \to \mathcal{A}_2$
of degree $40$, which is the one induced by the deformations of the Albanese map $\alpha \colon S \to \mathrm{Alb}(S)$. Observe that such a morphism is not surjective, because its image does not contain the products of elliptic curves that are not isomorphic to Jacobians.
\end{remark}


\bigskip
\bigskip

Roberto Pignatelli \\ Dipartimento di Matematica, Universit\`{a} di
Trento \\ Via Sommarive, 14 I-38123 Trento (TN), Italy.  \\
\emph{e-mail address:} \verb|Roberto.Pignatelli@unitn.it|

\bigskip

Francesco Polizzi \\ Dipartimento di Matematica e Informatica,
Universit\`{a} della
Calabria \\ Cubo 30B, 87036 Arcavacata di Rende (Cosenza), Italy.\\
\emph{E-mail address:} \verb|polizzi@mat.unical.it|

\end{document}